\documentclass[a4paper,11pt]{article}

\usepackage{amsmath}
\usepackage{amssymb}
\usepackage{amsthm}
\usepackage{graphicx}
\usepackage{amscd}
\usepackage{epic, eepic}
\usepackage{url}
\usepackage{color}
\usepackage[utf8]{inputenc} 
\usepackage{enumerate}   
\usepackage{graphicx}
\usepackage{epstopdf}
\usepackage{enumitem}
\usepackage{tikz}
\usetikzlibrary{positioning,arrows,shapes,decorations.markings,decorations.pathreplacing,matrix,patterns,bending}
\usetikzlibrary{arrows.meta}
\tikzstyle{vertex}=[circle,draw=black,fill=black,inner sep=0,minimum size=5pt,text=white,font=\footnotesize]
\tikzset{
	on each segment/.style={
		decorate,
		decoration={
			show path construction,
			moveto code={},
			lineto code={
				\path [#1]
				(\tikzinputsegmentfirst) -- (\tikzinputsegmentlast);
			},
			curveto code={
				\path [#1] (\tikzinputsegmentfirst)
				.. controls
				(\tikzinputsegmentsupporta) and (\tikzinputsegmentsupportb)
				..
				(\tikzinputsegmentlast);
			},
			closepath code={
				\path [#1]
				(\tikzinputsegmentfirst) -- (\tikzinputsegmentlast);
			},
		},
	},
}

\usepackage[hidelinks]{hyperref}
\usepackage[top=25mm, bottom=25mm, left=28mm, right=28mm]{geometry}

\makeatletter
\renewenvironment{proof}[1][\proofname] {\par\pushQED{\qed}\normalfont\topsep6\p@\@plus6\p@\relax\trivlist\item[\hskip\labelsep\bfseries#1\@addpunct{.}]\ignorespaces}{\popQED\endtrivlist\@endpefalse}
\makeatother

\newtheorem{theorem}{\bf Theorem}[section]
\newtheorem{lemma}[theorem]{\bf Lemma}

\newtheorem{conjecture}[theorem]{\bf Conjecture}

\theoremstyle{definition}

\def\eps{\varepsilon}

\title{Long directed paths in Eulerian digraphs}
\author{Oliver Janzer\thanks{Department of Mathematics, ETH Z\"urich, Switzerland.\newline 
E-mail: {\tt \{oliver.janzer,benjamin.sudakov,istvan.tomon\}@math.ethz.ch}.}
\and
Benny Sudakov\footnotemark[1]
\and
Istv\'an Tomon\footnotemark[1]}
\date{}

\begin{document}

\maketitle

\begin{abstract}
    An old conjecture of Bollob\'as and Scott asserts that every Eulerian directed graph with average degree $d$ contains a directed cycle of length at least $\Omega(d)$. The best known lower bound for this problem is $\Omega(d^{1/2})$ by Huang, Ma, Shapira, Sudakov and Yuster. They asked whether this estimate can be improved at least for directed paths instead of cycles and whether one can find a long path starting from any vertex if the host digraph is connected. In this paper we break the $\sqrt{d}$ barrier, showing  how to find a path of length $\Omega(d^{1/2+1/40})$ from any vertex of a connected Eulerian digraph.
\end{abstract}

\section{Introduction}

Determining the maximum possible number of edges in a graph which does not contain certain graphs as a subgraph is a central problem in Extremal Combinatorics. In the case where a path of fixed length is forbidden, this problem is well understood. Let us write $P_k$ for the path with $k$ edges. A well-known result of Erd\H os and Gallai \cite{EG59} states that any graph on $n$ vertices with average degree more than $k-1$ contains a copy of $P_k$. This is tight when $k-1$ divides $n$ since the disjoint union of $\frac{n}{k-1}$ cliques of size $k-1$ does not contain a path of length $k$. In the same paper it was also shown that any $n$-vertex graph with more than  $\frac{1}{2}(k-1)(n-1)$ edges contains a cycle of length at least $k$. This is tight when $k-2$ divides $n-1$ since the union of $\frac{n-1}{k-2}$ cliques of size $k-1$ sharing a single vertex does not contain a cycle of length at least $k$. The maximum number of edges in an $n$-vertex graph without a path of length $k$ was determined for every $n,k$ by Faudree and Schelp \cite{FS75} and, independently, by Kopylov \cite{Ko77}. The maximum possible number of edges in an $n$-vertex graph without a cycle of length at least $k$ was determined for all $n,k$ independently by Woodall~\cite{Wo76} and Kopylov~\cite{Ko77}.

In this paper we investigate how long a path there must exist in a directed graph (digraph) with a given number of edges. The edges of our directed graphs are ordered pairs of vertices. We use the convention that loops and multiple edges are not allowed, but we permit the possibility of having both $(u,v)$ and $(v,u)$ as an edge. We write $uv$ as a shorthand for $(u,v)$.

Observe that, very much unlike in the graph case, a directed path of length at least two need not exist even when our host graph is a dense directed graph. Indeed, if $G$ is the digraph whose vertex set is $A\cup B$ and the edges of $G$ are $uv$ for all $u\in A,v\in B$, then $G$ has no directed path of length at least two. Hence, in order to be able to find longer paths, one needs to make further assumptions about the directed graph. One natural assumption that takes care of the above example is that the indegree and the outdegree agree at every vertex. Write $d^+(v)$ and $d^-(v)$ for the outdegree and indegree of a vertex $v$, respectively. If the digraph $G$ satisfies $d^+(v)=d^-(v)$ for every vertex $v\in V(G)$, we call $G$ Eulerian. The degree of $v$ is then defined to be $d^+(v)=d^-(v)$. Note that, somewhat unconventionally, we do not require an Eulerian directed graph to be connected.

Twenty-five years ago, Bollob\'as and Scott proposed the following intriguing question about directed cycles.

\begin{conjecture}[Bollob\'as--Scott \cite{BS96}] \label{con:longcycle}
    Let $G$ be an Eulerian directed graph with average degree at least $d$. Then $G$ contains a directed cycle of length at least $cd$ for some absolute constant $c>0$.
\end{conjecture}

This conjecture is still wide open. The best known lower bound is $\Omega(d^{1/2})$, proved by Huang, Ma, Shapira, Sudakov and Yuster \cite{HMSSY}. They deduce this bound from the related problem of finding an Eulerian subgraph with large minimum degree. Since the bound in \cite{HMSSY} for this latter problem is tight, $\Omega(d^{1/2})$ is the best estimate one can get by their approach.  In the case we also assume that the maximum degree of $G$, denoted by $\Delta(G)$, is not too large, this is improved by Knierim, Larcher, Martinsson and Noever \cite{KLMN19}, who prove that $G$ contains a directed cycle of length $\Omega(d/\log \Delta(G))$. However, $\Delta(G)$ can be arbitrarily large compared to $d$ (even linear in the number of vertices $n$). Hence this result does not seem to imply any bound as a function of $d$ alone.

In \cite{HMSSY}, it is noted that the relaxation of the conjecture of Bollob\'as and Scott, in which we look for a directed path instead of a cycle, is open as well, with the best known lower bound still $\Omega(d^{1/2})$. Moreover, in \cite{HMSSY}, they propose the conjecture that if we assume that the digraph $G$ is connected as well, then $G$ contains a directed path of length $\Omega(d)$ starting from any vertex. The main result of our paper is the following theorem, which breaks the  $\Omega(d^{1/2})$ barrier for paths.

\begin{theorem} \label{thm:directedpath}
    Let $\eps=1/40$. Then there exists a constant $c>0$ such that the following holds. Let $G$ be an Eulerian, connected directed graph with average degree at least $d$ and let $v\in V(G)$. Then $G$ contains a directed path of length at least $cd^{1/2+\eps}$ starting at $v$.
\end{theorem}

We call a directed graph \emph{strongly connected} if for any two vertices $u,v$ there exists a path from $u$ to $v$. Note that if $G$ is Eulerian, then $G$ is strongly connected if and only if the underlying unordered graph is connected.   If $G$ is an Eulerian digraph, then every component of $G$ is Eulerian as well, and if $G$ has average degree $d$, then $G$ contains a component of average degree at least $d$. Therefore, Theorem \ref{thm:directedpath} immediately implies that any Eulerian digraph of average degree $d$ contains a directed path of length $\Omega(d^{1/2+1/40})$.


\section{Overview of the proof}

Let us give a brief outline of the proof of Theorem \ref{thm:directedpath}. We will proceed by induction on $\lfloor d\rfloor$, and let $\phi(d)=cd^{1/2+\eps}$. Let $G$ be a connected Eulerian digraph on $n$ vertices of average degree at least $d$, and let $v\in V(G)$. 

Suppose that $G$ contains no directed path of length at least $\phi(d)$ starting at $v$. Then it is easy to show that $G$ contains no directed cycle of length at least $\phi(d)+1$. As $G$ is Eulerian, we can partition the edge set of $G$ into cycles. We remove the edges of those cycles in this partition that contain $v$. The resulting graph $G'$ is still Eulerian, so if the density of $G'$ is at least $d'$ for some $d'\leq d-1$, then $G'$ contains a directed path of length at least $\phi(d')$. If $d'$ is not too small, more precisely $\phi(d')\geq \phi(d)-1$, then this path can be extended to a path starting from $v$ of length at least $\phi(d)$. However, if $d'$ is small, then the degree of $v$ must be very large, in particular of order $\Omega(nd^{-2\eps})$.


We shall prove that there cannot be a small set $S$ (of size significantly less than $d^{1/2}$) containing $v$ such that $G-S$ is not strongly connected. Indeed, if there is such a set, we can find a long path from $v$ as follows. Take a cycle decomposition of the edge set of $G$ and remove those cycles which go through $S$. Then we obtain an Eulerian subgraph $G'$ of $G$. It is disconnected, since $G-S$ is not strongly connected. Since $S$ is not too large, and each cycle that we removed is not too long, the average degree of $G'$ is close to $d$. Hence, there exists a connected component $T$ in $G'$ with average degree $d'$ close to $d$. By the induction hypothesis, from any vertex of $T$ there is a path of length at least $\phi(d')$ in $G'\lbrack T\rbrack$. Thus, it suffices to find a path in $G$ starting at $v$ and ending in $T$ which only meets $T$ in its endvertex and which has length at least $\phi(d)-\phi(d')$. We find such a path by going from $v$ to another dense connected component $R$ of $G'$, taking a fairly long path inside $R$, and going from $R$ to $T$. See Lemma \ref{lemma:highlyconnected} for the details.

Thus, there is no set of size roughly $d^{1/2}$ whose removal makes $G$ not strongly connected. Hence, we can easily find a path $P$ of length close to $d^{1/2}$ starting at $v$ such that $G-A$ is strongly connected, where $A$ is the set of vertices of $P$, except for the last vertex. The average degree of the vertices on $P$ must be very large, close to linear in $n$. Indeed, take a cycle decomposition of the edges of $G$ and remove all those cycles from $G$ which go through $V(P)$. Call the resulting subgraph $G'$. If the average degree of the vertices on $P$ is not large in $G$, then we removed only a few short cycles, so $G'$ has large average degree. Then by the induction hypothesis we can find a path in $G'$ with length close to $\phi(d)$. Using that $G-A$ is strongly connected, we can turn this into a path from $v$ with length at least $\phi(d)$ by using $P$ as an initial segment.

Therefore we have a long path $P$ with $G-A$ strongly connected such that the vertices on $P$ have large average degree. Similarly as before, we intend to use $P$ as the initial segment of the long path from $v$ that we seek for. To extend $P$ to a long path, we would like to find an Eulerian graph $G'$ with vertex set $V(G)\setminus V(P)$ which has average degree close to $d$. To find such a $G'$, we need to remove the edges of $G$ incident to $V(P)$ but keep the graph Eulerian. Since the vertices on $P$ have large average degree, we would lose too many edges if we removed all cycles through $V(P)$, so we shall do something slightly different. As long as there are vertices $u,v\in V(G)\setminus V(P)$ with $N^+(u)\cap N^-(v)\cap V(P)$ large, we choose some $w\in N^+(u)\cap N^-(v)\cap V(P)$, we add a new edge $uv$ and remove edges $uw$ and $wv$. This way we keep our graph Eulerian and get rid of the edges incident to $V(P)$. The remaining graph has average degree $d'$ close to $d$, and we can find in it a path of length at least $\phi(d')$. However, we may have some edges in this path which were not present in $G$. To get a valid path in $G$, we replace any such fake edge $uv$ with a path $uwv$ such that $w\in N_G^+(u)\cap N_G^-(v)\cap V(P)$. By the definition of the new (fake) edges, there are many possible choices for $w$, so we can make sure that the choices for different pairs $(u,v)$ are distinct. With some more delicate argument, we will make sure that we can modify the path $P$ very slightly such that it avoids these vertices $w$ (which we should not use more than once in our path). Combining this modified version of $P$ with the path of length at least $\phi(d')$ that we found (and using that $G-A$ is strongly connected to connect them), we get a path of length at least $\phi(d)$ starting at $v$.

\section{The key lemmas and some preliminaries}

Our notation is mostly standard. For a directed graph $G$ and vertex $v\in V(G)$, we write $N_G^+(v)$ for the set of outneighbours of $v$, and $N_G^-(v)$ for the set of inneighbours of $v$ in $G$. We omit floor and ceiling signs whenever they are not crucial.

In what follows, fix $\eps=1/40$ and suppose that $c>0$ is sufficiently small. Also, define the function $\phi(t)=ct^{1/2+\eps}$. 

If $0\leq d\leq 1$, then Theorem \ref{thm:directedpath} trivially holds. Fix $d>1$,  and assume that Theorem \ref{thm:directedpath} holds for every Eulerian digraph $G_0$ with average degree at least $d_{0}\leq d-1$. Then, we prove that Theorem \ref{thm:directedpath} holds for every Eulerian digraph $G$ with average degree at least $d$ as well.  This will follow easily from the next two lemmas.

Say that an Eulerian graph $G$ is \emph{$d$-full}, if $G$ has at least $(|V(G)|-1)d$ edges, but no proper Eulerian subgraph $H\subset G$ has at least 2 vertices and at least $(|V(H)|-1)d$ edges. Also, for simplicity, write $\gamma:=d^{\eps}$.

\begin{lemma} \label{lemma:findgoodpath}
   Let $G$ be an Eulerian, connected and $d$-full digraph on $n\geq 2$ vertices. Let $v\in V(G)$. Assume that $G$ has no path of length at least $\phi(d)$ starting at $v$. Then there exists a path $P$ in $G$ starting at $v$ with the following properties.
    
    \begin{enumerate}
        \item $P$ has length $p\geq d^{1/2}\gamma^{-5}$. \label{plong}
        
        \item \label{pconnected} Let $A$ be the set of vertices of $P$, except its endvertex. Then $G-A$ is strongly connected.
        
        \item Let $Y$ be the set of vertices of $P$ which have degree at least $\gamma^{-2}n$ in $G$. List them as $Y=\{y_1,\dots,y_{\ell}\}$ in their order of appearance on $P$. Then for all but at most $\gamma^{-2}|Y|+\gamma^{7}$ vertices $y_i\in Y$, $G$ has an edge of the form $y_jy_{j'}^+$ with $i-\gamma^{9}\leq j<i<j'\leq i+\gamma^{9}$. Here and below, for a vertex $x\in V(P)$ which is not the endvertex of $P$, $x^+$ denotes the vertex following $x$ on $P$. \label{prich}
        
    \end{enumerate}
\end{lemma}

\begin{lemma} \label{lemma:finishproof}
   Let $G$ be an Eulerian, connected and $d$-full digraph on $n\geq 2$ vertices. Let $v\in V(G)$. Let $P$ be a path in $G$ starting at $v$ and satisfying properties \ref{plong}, \ref{pconnected}, \ref{prich} in the previous lemma. Then $G$ has a path of length at least $\phi(d)$ which starts at $v$.
\end{lemma}

\begin{proof}[Proof of Theorem \ref{thm:directedpath}]
 Let $G$ be an Eulerian, connected directed graph with average degree at least $d$ and let $v\in V(G)$. Choose a minimal Eulerian subgraph $H$ with at least two vertices and at least $(|V(H)|-1)d$ edges. It is easy to see that $H$ must be connected. Since $G$ is connected, there is a path from $v$ to $V(H)$ in $G$. We may choose such a path which only meets $V(H)$ in its last vertex $w$. By Lemma \ref{lemma:findgoodpath} and Lemma \ref{lemma:finishproof}, $H$ has a path of length at least $\phi(d)$ starting at $w$. Together with the path from $v$ to $w$ we found earlier, we obtain a path of length at least $\phi(d)$ in $G$ starting at $v$.
\end{proof}

We shall give the proofs of Lemma \ref{lemma:findgoodpath} and Lemma \ref{lemma:finishproof} in the next section. We now state and prove a few more results which will be used in the proofs. Let us start with the following simple lemma.

\begin{lemma}\label{lemma:shortcycle}
Let $G$ be a strongly connected digraph, and let $v\in V(G)$. Suppose that $G$ contains no path of length $\phi(d)$ which starts at $v$. Then $G$ contains no directed cycle of length at least $\phi(d)+1$. 
\end{lemma}

\begin{proof}
Assume that $G$ does contain a cycle of length at least $\phi(d)+1$. If the cycle contains $v$, then we trivially have a path of length at least $\phi(d)$ starting at $v$. Otherwise, as $G$ is strongly connected, there exists a path from $v$ to the vertex set of our cycle. Using this path and the cycle, we get a path of length at least $\phi(d)$ starting at $v$.
\end{proof}

We note that the function $\phi$ satisfies the following inequality.

\begin{lemma} \label{lemma:simpleineq}
    Let $0\leq \lambda\leq \frac{1}{4}d^{1/2}\gamma$, then 
    $$\phi(d-\lambda d^{1/2}\gamma^{-1})\geq \phi(d)-2c\lambda.$$
\end{lemma}

\begin{proof}
    Note that for $0\leq x\leq 1/4$, we have $1-x\geq e^{-2x}$. Hence $(1-x)^r\geq e^{-2rx}\geq 1-2rx$ for any $r>0$. Therefore, taking $x=\lambda d^{-1/2-\eps}$ and $r=1/2+\eps$,
    \begin{align*}
        (d-\lambda d^{1/2-\eps})^{1/2+\eps}&=d^{1/2+\eps}(1-\lambda d^{-1/2-\eps})^{1/2+\eps}\\
        &\geq d^{1/2+\eps}(1-2(1/2+\eps)\lambda d^{-1/2-\eps})\\
        &\geq d^{1/2+\eps}-2\lambda.
    \end{align*}

\end{proof}
The following lemma shows that if there exists no path of length $\phi(d)$ from $v$, then the degree of $v$ must be very large.

\begin{lemma}\label{lemma:bigdeg}
    Let $G$ be a connected, $d$-full, Eulerian digraph on $n\geq 2$ vertices. Let $v\in V(G)$. Suppose that $G$ contains no path of length at least $\phi(d)$ starting at $v$. Then the degree of $v$ is at least $n\gamma^{-2}$.
\end{lemma}

\begin{proof}
By Lemma \ref{lemma:shortcycle}, every cycle in $G$ has length at most $\phi(d)+1\leq 2\phi(d)$ (if $\phi(d)<1$, then it is trivial to find a path of length at least $\phi(d)$). Suppose that the degree of $v$ is less than $\gamma^{-2}n$. Partition the edge set of $G$ into cycles and remove those which go through $v$, then we get a graph with at least $(n-1)d-\gamma^{-2}n\cdot 2\phi(d)\geq (n-1)(d-4cd^{1/2}\gamma^{-1})$ edges. In this graph $\tilde{G}$ there exists a connected component of average degree at least $d-4cd^{1/2}\gamma^{-1}$, in which from every vertex there is a path of length at least $\phi(d-4cd^{1/2}\gamma^{-1})\geq \phi(d)-1$ in $\tilde{G}$, where the inequality holds by Lemma \ref{lemma:simpleineq}, since $8c^2 \leq 1$ for small enough $c$. (If $d-4cd^{1/2}\gamma^{-1}>d-1$, then we cannot assume the existence of a path of length at least $\phi(d-4cd^{1/2}\gamma^{-1})$ in the dense connected component, but then we have a path of length at least $\phi(d-1)\geq \phi(d)-1$. In the rest of the paper we will often have a similar situation where for small values of $d$ a certain formula for the average degree of our subgraph may be greater than $d-1$ and hence the induction hypothesis for Theorem \ref{thm:directedpath} would not apply, but in this case, like above, we can just use the theorem for $d-1$.)

Since $G$ is strongly connected, we get a path of length at least $\phi(d)$ in $G$ starting at $v$.
\end{proof}

Next, we find a long cycle assuming that the host graph has not too large maximum degree compared to $d$. The following lemma is a weaker version of the result of Knierim et al. \cite{KLMN19} mentioned in the introduction. Nevertheless, we give a short proof to make our paper self-contained and to illustrate that bounding the maximum degree of a digraph indeed simplifies the problem.

\begin{lemma} \label{lemma:fewvertices}
    Let $d$ be sufficiently large and let $G$ be an Eulerian directed graph with average degree at least $d$ and maximum degree $\Delta(G)\leq d^{20}$. Then $G$ has a cycle of length at least $\frac{d^{2/3}}{400\log d}$.
\end{lemma}

\begin{proof}
Let $\mathcal{C}$ be a maximal collection of pairwise edge-disjoint cycles of length at most $d^{1/3}$ in $G$. Let $G_1$ be the subgraph of $G$ whose edge set is $\bigcup_{C\in \mathcal{C}} C$. Let $G_2$ be the subgraph of $G$ whose edge set is $E(G)\setminus E(G_1)$. Clearly, both $G_1$ and $G_2$ are Eulerian. Moreover, $G_2$ contains no cycle of length at most $d^{1/3}$. Since $G_1$ and $G_2$ partition the edge set of $G$, we have $e(G_1)\geq nd/2$ or $e(G_2)\geq nd/2$.
\begin{description}
\item[Case 1.] $e(G_1)\geq nd/2$.

Delete vertices and edges of $G_1$ as follows. As long as $G_1$ has a vertex $v$ of degree at most $d^{2/3}/4$, delete $v$ and the edges of all $C\in \mathcal{C}$ which contain $v$. In each step we delete a vertex, we delete at  most $\frac{d^{2/3}}{4}\cdot d^{1/3}=d/4$ edges, so in total we deleted at most $nd/4$ edges. Moreover, the resulting subgraph of $G_1$ is non-empty and is a disjoint union of cycles, so it is Eulerian. Hence, $G_1$ has an Eulerian subgraph $H$ with minimum degree at least $d^{2/3}/4$. It is easy to see that $H$ contains a cycle of length at least $d^{2/3}/4$.

\item[Case 2.] $e(G_2)\geq nd/2$.

Assume that every cycle in $G_2$ has length at most $d^{2/3}$, otherwise we are done. By a similar procedure as in Case 1, we can find an Eulerian subgraph $H$ in $G_2$ which has minimum degree at least $d^{1/3}/4$. Let $t=\frac{d^{1/3}}{200\log d}$ and partition $V(H)$ randomly as $S_1\cup S_2\cup \dots \cup S_t$ by taking $v\in S_i$ with probability $1/t$ for every $v\in V(H)$, $1\leq i\leq t$, independently for all $v$. For any $v\in V(H)$ and any $1\leq i\leq t$, the probability that $N^+_H(v)\cap S_i=\emptyset$ is at most $(1-1/t)^{d^+_H(v)}\leq \exp(-\frac{d^+_H(v)}{t})\leq \exp(-\frac{d^{1/3}}{4t})=d^{-50}$. Note that for any fixed $v\in V(H)$ and $1\leq i\leq t$, the event $N_H^+(v)\cap S_i=\emptyset$ is independent of the events $N_H^+(w)\cap S_j=\emptyset$ for all $w\in V(H)$ with $N_H^+(w)\cap N_H^+(v)=\emptyset$ and all $j$. This means that $N_H^+(v)\cap S_i=\emptyset$ is independent of all but at most $\Delta(H)^2\cdot t$ events of the form $N_H^+(w)\cap S_j=\emptyset$. But $\Delta(H)^2\cdot t< d^{50}/e$, so by the Lov\'asz Local Lemma, with positive probability, for every $v\in V(H)$ and every $1\leq i\leq t$ we have $N^+_H(v)\cap S_i\neq \emptyset$. We now prove that if this holds, then $H$ contains a cycle of length at least $d^{2/3}/(400\log d)$. Define a sequence of vertices $v_1,v_2,\dots$ recursively as follows. Choose $v_1\in S_1$ arbitrarily. Suppose we have already defined $v_1,\dots,v_i$. Let $v_i\in S_j$. If $v_{i-1},v_{i-2},\dots,v_{i-d^{1/3}}\in S_j$, then let $v_{i+1}$ be a neighbour of $v_i$ in $S_{j+1}$ (where $S_{t+1}:=S_1$). Otherwise, let $v_{i+1}$ be a neighbour of $v_i$ in $S_j$. Since $H$ does not contain a cycle of length at most $d^{1/3}$, it follows that $v_{i+1}\not \in \{v_k: i-d^{1/3}\leq k\leq i-1\}$. Let $b$ be the smallest positive integer for which there exists some $a<b$ with $v_a=v_b$. It is not hard to see that $b-a\geq (t-1)d^{1/3}$. Since $d$ is sufficiently large, this is at least $\frac{d^{2/3}}{400\log d}$, which completes the proof.
\end{description}
\end{proof}

\begin{lemma} \label{lemma:settovertex}
    Let $G$ be an Eulerian, connected directed graph on $n$ vertices. Assume that $G$ has average degree $k$, the longest cycle in $G$ has length $t$ and that $G$ has no Eulerian subgraph with average degree greater than $d$. Let $B\subset V(G)$ be a vertex set of size $\beta n$ and let $v\in V(G)$. Then $G$ has a path of length at least $\frac{k-(1-\beta)d}{t}-1$ which starts in $B$ and ends at $v$.
\end{lemma}

\begin{proof}
    For convenience, we will instead prove the equivalent statement that $G$ has a path of length at least $\frac{k-(1-\beta)d}{t}-1$ which starts at $v$ and ends in $B$. (The equivalence can be seen by reversing each arrow of $G$.)

    Let $G_0=G$, $S_0=V(G)$ and $x_0=v$. Since $G_0$ is Eulerian, we can decompose its edge set to a set of cycles. Take one such decomposition and remove the edges of those cycles which go through $x_0$. Remove, in addition, vertex $x_0$. Call the resulting graph $G_1$.
    
    Let $S_{1,1},S_{1,2},\dots,S_{1,q_1}$ be the connected components of $G_1$. For every $1\leq j\leq q_1$, choose a path in $G_0$ which starts at $x_0$ and whose only vertex in $S_{1,j}$ is the endvertex of the path. Call this endvertex $v_{1,j}$. $G_1$ is Eulerian, so there exists a decomposition of its edges into cycles. Take one such decomposition and remove all edges of the cycles which go through any of $v_{1,1},\dots,v_{1,q_1}$. Remove, in addition, the vertices $v_{1,1},\dots,v_{1,q_1}$. Call the resulting graph $G_2$.    Let $S_{2,1},S_{2,2},\dots,S_{2,q_2}$ be the connected components of $G_2$. For every $1\leq j\leq q_2$, choose $j'$ such that $S_{2,j}\subset S_{1,j'}$. Since $G_1[S_{1,j'}]$ is connected and contains $v_{1,j'}$, there exists a path in $G_1$ which starts at $v_{1,j'}$ and whose only vertex in $S_{2,j}$ is its endvertex. Call this endvertex $v_{2,j}$.    Repeat the above procedure to get graphs $G_3,G_4,\dots$ and stop when we get a graph $G_p$ with $B\cap V(G_p)=\emptyset$.
    
    Observe first that for any $i$ and any $y\in V(G_i)$, there exists a path of length at least $i$ from $x_0$ to $y$ in $G$. Thus, since $B\cap V(G_{p-1})\neq \emptyset$, in order to prove the lemma we just need to show that $p-1\geq \frac{k-(1-\beta)d}{t}-1$.   Since $B\cap V(G_p)=\emptyset$, we have $|V(G_p)|\leq (1-\beta)n$. But $G_p$ is an Eulerian subgraph of $G$, so it has average degree at most $d$. Hence,
    \begin{equation}
        e(G_p)\leq (1-\beta)nd. \label{eqn:edgelowerbound}
    \end{equation}
    
    On the other hand, observe that for every $i$, we have $e(G_{i+1})\geq e(G_{i})-nt$. Indeed, to get $G_{i+1}$ from $G_{i}$, we remove at most $n$ cycles (we remove at most $|S_{i,j}|$ cycles through $v_{i,j}$ in each connected component $S_{i,j}$ of $G_i$). But each cycle of $G$ has length at most $t$, so we remove at most $nt$ edges in total. Thus, $e(G_p)\geq e(G_0)-pnt=e(G)-pnt= nk-pnt$. Comparing this to equation (\ref{eqn:edgelowerbound}), we get $(1-\beta)nd\geq nk-pnt$, so $p\geq \frac{k-(1-\beta)d}{t}$.
\end{proof}

The next result will be used to prove Lemma \ref{lemma:findgoodpath}.

\begin{lemma} \label{lemma:highlyconnected}
   Let $G$ be an Eulerian, $d$-full, connected digraph on $n\geq 2$ vertices. Let $v\in V(G)$ and let $S\subset V(G)$ be a set of size $s\leq d^{1/2}\gamma^{-5}$ containing $v$. Suppose that $G$ has no path of length at least $\phi(d)$ starting at $v$. Then $G-S$ is strongly connected.
\end{lemma}

In order to prove this lemma, we first need to establish the following weaker version.

\begin{lemma} \label{lemma:oneconnected}
   Let $G$ be an Eulerian, $d$-full, connected digraph on $n\geq 2$ vertices. Let $v\in V(G)$. Suppose that $G$ has no path of length at least $\phi(d)$ starting at $v$. Then $G-v$ is strongly connected.
\end{lemma}

\begin{proof}

By Lemma \ref{lemma:shortcycle}, $G$  contains no cycle of length at least $\phi(d)+1$. Now define an auxiliary directed graph $K$ as follows. Note that the vertex set of every directed graph can be partitioned into maximal subsets in a unique way such that each of the sets induces a strongly connected subgraph, call these the \emph{strongly connected components} of the digraph. Let the vertices of $K$ be the strongly connected components of $G-v$, and put a directed edge from component $T$ to component $R$ in $K$ if there is an edge from $T$ to $R$ in $G-v$. Note that $K$ is acyclic, otherwise the union of the components forming a cycle is a strongly connected subset of $G-v$.

Assume that $G-v$ is not strongly connected. We claim that there is no isolated vertex in $K$. Indeed, if there was one, say component $T$ of $G-v$, then any in- or outneighbour (in $G$) of any vertex in $T$ is in the set $T\cup \{v\}$. Thus, the subgraph $G[T\cup \{v\}]$ is Eulerian. Since it has at least two vertices but $T\cup \{v\}\neq V(G)$, by the $d$-fullness of $G$ we have $e(G[T\cup \{v\}])< |T|d$. Similarly, $G[V(G)\setminus T]$ is Eulerian, so $e(G[V(G)\setminus T])< (n-|T|-1)d$. Hence, $e(G)=e(G[T\cup \{v\}])+e(G[V(G)\setminus T])<(n-1)d$, which is a contradiction. Thus, $K$ contains no isolated vertex, so it is non-empty. Choose a path in $K$ of maximal length and let its endpoints be $T$ and $R$. Since $K$ is acyclic, $T$ has no incoming edge and $R$ has no outgoing edge in $K$.

Since $G$ is Eulerian, we can partition its edge set into cycles. Remove all cycles passing through $v$ to get the graph $G'$.

\medskip

\noindent \textbf{Claim.} Let $f$ be the number of edges in $G$ which start at $v$ and end in $T$. Then $e(G'[T])\geq |T|d-2\phi(d)f$.

\medskip

\noindent \textbf{Proof of Claim.} Consider that same partition of $E(G)$ into cycles that was used to define $G'$. Remove only those cycles which go through $v$ and intersect $T$. Call the resulting graph $G''$. Clearly $G''[T]=G'[T]$. Since $T$ has no inneighbour in $K$, it follows that any inneighbour of any vertex in $T$ is in $T\cup \{v\}$. So any cycle in $G$ through $v$ which intersects $T$ must contain an edge from $v$ to $T$. Hence, there are at most $f$ such edge-disjoint cycles. Since any cycle in $G$ has length at most $2\phi(d)$, we have $e(G'')\geq e(G)-2\phi(d)f$. Moreover, there is no edge from $V(G)\setminus T$ to $T$ in $G''$, and $G''$ is Eulerian, so there is no edge from $T$ to $V(G)\setminus T$ either. Assume that $e(G''[T])<|T|d-2\phi(d)f$. Then 
\begin{align*}
e(G''[V(G)\setminus T])&=e(G'')-e(G''[T])> e(G)-2\phi(d)f-(|T|d-2\phi(d)f)\\
&=e(G)-|T|d\geq (n-1)d-|T|d=(n-|T|-1)d.
\end{align*}
But $G''[V(G)\setminus T]$ is a proper Eulerian subgraph of $G$ with at least two vertices, so this contradicts that $G$ is $d$-full. $\Box$

\medskip

By the claim, if $f\leq \gamma^{-2}|T|$, then $e(G'[T])\geq |T|d-2cd^{1/2}\gamma^{-1}|T|$. So $G'[T]$ has average degree at least $d-2cd^{1/2}\gamma^{-1}$, hence it has a connected component $T'$ such that $G'[T']$ has average degree at least $d-2cd^{1/2}\gamma^{-1}$. Since $G$ is strongly connected, there exists a path from $v$ to $T'$ in $G$ which only intersects $T'$ in the endvertex $w$. Also, by our induction hypothesis, it follows that there exists a path in $G'[T']$ of length at least $\phi(d-2cd^{1/2}\gamma^{-1})$ starting at $w$. By Lemma \ref{lemma:simpleineq}, $\phi(d-2cd^{1/2}\gamma^{-1})\geq \phi(d)-4c^2$, which is at least $\phi(d)-1$ for $c<1/2$. Combining this with the path from $v$ to $w$, we get a path in $G$ of length at least $\phi(d)$ starting at $v$, which is a contradiction. 

So we may assume that $f\geq \gamma^{-2}|T|$. However, $f\leq |T|$, so the claim gives $e(G'[T])\geq |T|d-2\phi(d)|T|$. Similarly as in the above claim, we have $e(G'[R])\geq |R|d-2\phi(d)|R|$. Choose a connected component $R'$ of $G'[R]$ such that $G'[R']$ has average degree at least $d-2\phi(d)$. Let $W$ be the union of those connected components of $G'[T]$ in which the proportion of vertices that are outneighbours of $v$ is at least $\frac{1}{2}\gamma^{-2}$. Since $|N_G^+(v)\cap T|\geq \gamma^{-2}|T|$, we have $|W|\geq \gamma^{-2}|T|/2$. Note that $e(G'[T\setminus W])\leq (|T|-|W|)d$, so 
\begin{align*}
    e(G'[W])&=e(G'[T])-e(G'[T\setminus W])\geq |T|d-2\phi(d)|T|-(|T|-|W|)d\\
    &=|W|d-2\phi(d)|T|\geq |W|d-4\phi(d)\gamma^{2}|W|.
\end{align*} Hence, there exists a connected component $W'\subset W$ of $G'$ such that $G'[W']$ has average degree at least $d-4\phi(d)\gamma^{2}$ and $|N_G^+(v)\cap W'|\geq \frac{1}{2}\gamma^{-2}|W'|$. Let $B=N^+_G(v)\cap W'$.

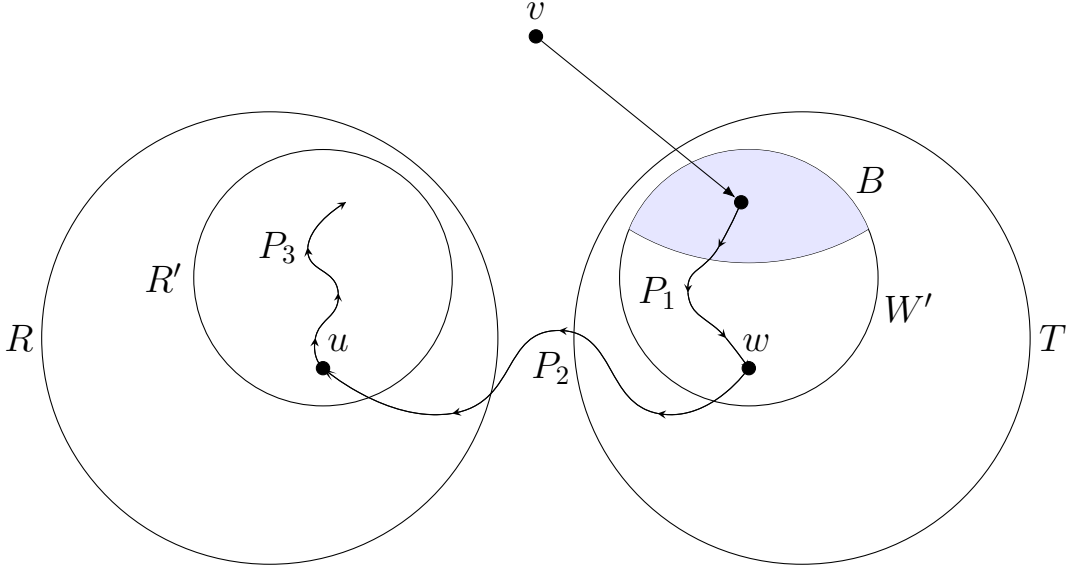
\begin{figure}
\centering
\begin{tikzpicture}

		\node[vertex, label=above:{\Large $v$}] (v) at (0,4) {};
		\draw (3.5,0) circle (3); \node at (6.8,0) {\Large $T$};
		\draw (2.8,0.8) circle (1.7);
		\node at (4.9,0.4) {\Large $W'$}; 
		\node[vertex, label={[shift={(0.1,0)}]:{\Large $w$}}] (w) at (2.8,-0.4) {};
		
		\draw (-3.5,0) circle (3); \node at (-6.8,0) {\Large $R$};
		\draw (-2.8,0.8) circle (1.7);
		\node at (-4.9,0.8) {\Large $R'$};
		\node[vertex, label={[shift={(0.2,0)}]:{\Large $u$}}] (u) at (-2.8,-0.4) {};

		\begin{scope}
			\clip  (2.8,0.8) circle (1.7);
			\draw  (2.8,4) circle (3);
			\fill[fill=blue!10!white] (2.8,4) circle (3);
		\end{scope}
		\node at (4.4,2.1) {\Large $B$}; 
		\node[vertex] (x) at (2.7,1.8) {};
		\draw[-{Latex}] (v) edge (x);

	\begin{scope}
		\draw plot [smooth,tension=1]
		coordinates {(x) (2.4,1.2) (2,0.6) (2.5,0.0) (w)}
		[postaction={on each segment={draw,-{stealth[bend]}}}];
		
		\draw plot [smooth,tension=1]
		coordinates {(w) (1.6,-1) (0.3,0.1) (-1.1,-1) (u)}
		[postaction={on each segment={draw,-{stealth[bend]}}}];
		
		\draw plot [smooth,tension=1]
		coordinates {(u)(-2.9,0) (-2.6,0.6)  (-3,1.2)  (-2.5,1.8)}
		[postaction={on each segment={draw,-{stealth[bend]}}}];
	\end{scope}

	\node at (1.6,0.6) {\Large $P_1$};
	\node at (0.2,-0.4) {\Large $P_2$};
	\node at (-3.4,1.2) {\Large $P_3$};

\end{tikzpicture}
\caption{An illustration for the case $f\geq \gamma^{-2}|T|$ in the proof of Lemma \ref{lemma:oneconnected}}
\label{figure1}
\end{figure}

There exists a path in $K$ from $T$ to $R$, so there also exists a path in $G-v$ from $T$ to $R$. Moreover, $T$ and $R$ are strongly connected components in $G-v$ and $W'\subset T$, $R'\subset R$, so there exists a path from $W'$ to $R'$ in $G-v$. Choose such a path $P_2$ of minimal length; then it only intersects $W'$ in the first vertex $w$ and only intersects $R'$ in the last vertex $u$. By Lemma \ref{lemma:settovertex}, $G'[W']$ has a path $P_1$ of length at least 
$$\frac{d-4\phi(d)\gamma^{2}-(1-\frac{1}{2}\gamma^{-2})d}{2\phi(d)}-1=
\frac{\frac{1}{2}d\gamma^{-2}-4\phi(d)\gamma^{2}}{2\phi(d)}-1\geq
\frac{1}{8c}d^{1/2}\gamma^{-3}$$ 
from $B$ to $w$. Finally, there exists a path $P_3$ in $G'[R']$ of length at least $\phi(d-2\phi(d))$ starting at $u$. Since every element of $B$ is an outneighbour of $v$ in $G$, it follows that $G$ has a path of length at least $|P_1|+|P_2|+|P_3|$ starting at $v$. But $|P_1|\geq \frac{1}{8c}d^{1/2}\gamma^{-3}$ and by Lemma~\ref{lemma:simpleineq}, $|P_3|\geq \phi(d-2\phi(d))\geq \phi(d)-4c^2\gamma^{2}$, so $|P_1|+|P_3|\geq \phi(d)$, which is a contradiction. See Figure \ref{figure1} for an illustration of the proof of this case.
\end{proof}

\begin{proof}[Proof of Lemma \ref{lemma:highlyconnected}.]

By Lemma \ref{lemma:shortcycle}, every cycle in $G$ has length at most $\phi(d)+1\leq 2\phi(d)$, and by Lemma \ref{lemma:bigdeg}, the degree of $v$ is at  least $\gamma^{-2}n$. Partition the edge set of $G$ into cycles and remove those which intersect $S$. Remove also the vertices in $S$. Let the resulting graph be $G'$. Since any removed cycle contains an edge incident to $S$ and any cycle has length at most $2\phi(d)$, we have $$e(G')\geq e(G)-|S|n\cdot 2\phi(d)\geq (n-1)d-2cnd\gamma^{-4}.$$

Let $W$ be the union of those connected components of $G'$ in which the proportion of vertices that are outneighbours of $v$ is at least $\gamma^{-2}/2$. Since $|N_G^+(v)|\geq \gamma^{-2}n$, we have $|W|\geq \gamma^{-2}n/2$. Note that $e(G'[V(G')\setminus W])\leq (n-1-|W|)d$, so 
\begin{align*}
e(G'[W])&=e(G')-e(G'[V(G')\setminus W])\geq (n-1)d-2cnd\gamma^{-4}-(n-1-|W|)d\\
&=|W|d-2cnd\gamma^{-4}\geq |W|d-4c|W|d\gamma^{-2}.
\end{align*}

Hence, there exists a connected component $W'\subset W$ of $G'$ such that $G'[W']$ has average degree at least $d-4cd\gamma^{-2}$ and $|N_G^+(v)\cap W'|\geq \frac{1}{2}\gamma^{-2}|W'|$.

Define an auxiliary directed graph $K$ as follows. The vertices of $K$ are the strongly connected components of $G-S$ and there is a directed edge from $T$ to $R$ in $K$ if there is an edge from $T$ to $R$ in $G-S$. Similarly as in Lemma \ref{lemma:oneconnected}, $K$ is acyclic. Assume that $G-S$ is not strongly connected. Then $K$ has at least two vertices, so since $K$ is acyclic, $K$ has at least two vertices which do not have both an incoming and an outgoing edge. We claim that if $T$ is a strongly connected component of $G-S$ which has no incoming (resp. outgoing) edge, then $e(G')\geq |T|d-2\phi(d)|T|\cdot|S|$. The proof of this is very similar to the proof of the claim from Lemma \ref{lemma:oneconnected}, so it is omitted. Since $G'[T]$ has average degree at least $d-2\phi(d)|S|$, it has a connected component of average degree at least $d-2\phi(d)|S|$.

Thus, $G'$ has at least two connected components with average degree at least $d-2\phi(d)|S|$. Choose such a component $U$ which is distinct from $W'$. By Lemma \ref{lemma:oneconnected}, the graph $G-v$ is strongly connected, so there exists a path from $W'$ to $U$ in $G-v$ which only meets $W'$ and $U$ at the first and last vertex, respectively. Call these vertices $w$ and $u$. Let $B=N^+_G(v)\cap W'$. By Lemma \ref{lemma:settovertex}, $G'[W']$ contains a path of length at least $$\frac{d-4cd\gamma^{-2}-(1-\frac{1}{2}\gamma^{-2})d}{2\phi(d)}-1\geq \frac{1}{8c}d^{1/2}\gamma^{-3}$$ which starts in $B$ and ends at $w$. Moreover, since $G'[U]$ has average degree at least $d-2\phi(d)|S|\geq d-2cd\gamma^{-4}$, it contains a path of length at least $\phi(d-2cd\gamma^{-4})$ starting at $u$. It follows, similarly as in Lemma \ref{lemma:oneconnected}, that $G$ has a path of length at least $$\frac{1}{8c}d^{1/2}\gamma^{-3}+\phi(d-2cd\gamma^{-4})$$ starting at $v$. But by Lemma \ref{lemma:simpleineq}, $$\phi(d-2cd\gamma^{-4})\geq \phi(d)-4c^2d^{1/2}\gamma^{-3}\geq \phi(d)-\frac{1}{8c}d^{1/2}\gamma^{-3},$$ which is a contradiction.
\end{proof}

\section{The proofs of our key lemmas}

\begin{proof}[Proof of Lemma \ref{lemma:findgoodpath}]

Since $c$ is sufficiently small and it is assumed that there is no path of length at least $\phi(d)$, $d$ is bounded from below by a sufficiently large constant.
Let $v_0=v$, and let $p=d^{1/2}\gamma^{-5}$. Define $v_1,v_2,\dots,v_{p}$ recursively as follows. For $1\leq i\leq p$, let $v_i$ be a uniformly random element from $N^+_G(v_{i-1})\setminus \{v_0,\dots,v_{i-2}\}$. This is well-defined since if $N^+_G(v_{i-1})\subset \{v_0,\dots,v_{i-2}\}$, then $G-\{v_0,\dots,v_{i-2}\}$ is not strongly connected, which contradicts Lemma~\ref{lemma:highlyconnected}.

Let $P$ be the path $v_0v_1\dots v_{p}$. Property \ref{plong} is clear, and property \ref{pconnected} follows from Lemma \ref{lemma:highlyconnected}. It remains to prove that property \ref{prich} holds with positive probability. Let $Y=\{y_1,\dots,y_{\ell}\}$ be the set of vertices of $P$ which have degree at least $\gamma^{-2}n$ such that $y_i$ appears before $y_j$ on $P$ for $i<j$. Call a vertex $y_i\in Y$ \emph{dangerous} if there exists no index $a$ such that $i-\gamma^{9}\leq a<i$ and the number of $b$ satisfying $i<b\leq i+\gamma^{9}$ and $|N^+_{G}(y_a)\cap N^+_{G}(y_b)|\geq \frac{1}{100}\gamma^{-4}n$ is at least $4\gamma^{4.5}$.

\medskip

\noindent \textbf{Claim.} In any ``interval'' $\{y_i,y_{i+1},\dots,y_{i+\gamma^{9}}\}$, the number of dangerous vertices is at most~$\gamma^{7}$.

\medskip

\noindent \textbf{Proof of Claim.} Suppose that there exist $J\subset [i, i+\gamma^9]$ such that $|J|=\gamma^{7}$, $y_j$ is dangerous for $j\in J$, and $\max J-\min J\leq\gamma^{9}$.  Note that $|N_G^+(y)|\geq \gamma^{-2}n$ for any $y\in Y$. Hence, among any $10\gamma^{2}$ vertices of $Y$, there exist distinct $z$ and $z'$ with $|N^+_{G}(z)\cap N^+_{G}(z')|\geq \frac{1}{100}\gamma^{-4}n$. Indeed, otherwise the union of the out-neighbourhoods of those $10\gamma^{2}$ vertices would have size at least 
$$10\gamma^{2}\cdot \gamma^{-2}n-\binom{10\gamma^{2}}{2}\cdot \frac{1}{100}\gamma^{-4}n>n,$$ which is a contradiction. It follows by Tur\'an's theorem that in any subset of $Y$ of size $t\geq 100\gamma^{2}$, the number of pairs $(z,z')$ with $|N^+_{G}(z)\cap N^+_{G}(z')|\geq \frac{1}{100}\gamma^{-4}n$ is at least $\frac{1}{1000}\gamma^{-2}t^2$. Applying this to $\{y_{j}:j\in J\}$, we get that in this set there are at least $\frac{1}{1000}\gamma^{12}$ pairs $(y_j,y_{j'})$ with 
$j<j'$ such that  $|N^+_{G}(y_{j})\cap N^+_{G}(y_{j'})|\geq \frac{1}{100}\gamma^{-4}n$. Define an auxiliary directed graph $K$ whose vertex set is $\{y_j:j\in J\}$, and in which $y_{j}\rightarrow y_{j'}$ is an edge if and only if $j<j'$ and $|N^+_{G}(y_{j})\cap N^+_{G}(y_{j'})|\geq \frac{1}{100}\gamma^{-4}n$. Then $K$ has at least $\frac{1}{1000}\gamma^{12}$ edges. So $K$ has a vertex $y_{j}$ with outdegree at least $\frac{1}{1000}\gamma^{5}$ in $K$. It is easy to see that this means that $y_{j^{*}}$ is not dangerous, where $j^{*}>j$ is the smallest index such that $j^{*}\in J$. $\Box$

\medskip

By splitting the set of indices into disjoint intervals of length $\gamma^9$,  it follows from the claim that the number of dangerous vertices in $Y$ is at most $\gamma^{-2}|Y|+\gamma^{7}$.  Call the pair $(i,j)$ with $0\leq j<i\leq p$ bad if there are at least $\gamma^{4.5}$ values $k>i$ such that $|N_G^+(v_j)\cap N_G^+(v_k)|\geq \frac{1}{100}\gamma^{-4}n$, and moreover for the $\gamma^{4.5}$ smallest such values of $k$ we have $v_k^+\not \in N_G^+(v_j)$.

Note that $P$ has length $p=d^{1/2}\gamma^{-5}\leq \frac{1}{200}\gamma^{-4}n$. Thus, for any $j<k$ with $|N_G^+(v_j)\cap N_G^+(v_k)|\geq \frac{1}{100}\gamma^{-4}n$, we have that 
$|N_G^+(v_j)\cap N_G^+(v_k)\setminus P|\geq \frac{1}{200}\gamma^{-4}n$ and hence
the probability (conditional on the choices for $v_j$ and $v_k$) that $v_k^+\in N_G^+(v_j)$ is at least $\frac{1}{200}\gamma^{-4}$. Thus, for any fixed $(i,j)$, the probability that $(i,j)$ is bad is exponentially small in $\gamma^{0.5}$.

By the union bound, if $d$ is sufficiently large, then with positive probability no pair is bad. However, if this holds and $y_i\in Y$ is a vertex that is not dangerous, then $G$ has an edge of the form $y_jy_{j'}^+$ with $i-\gamma^{9}\leq j<i<j'\leq i+\gamma^{9}$. There are at most $\gamma^{-2}|Y|+\gamma^{7}$ dangerous vertices in $Y$, so the proof of property~\ref{prich} is complete.
\end{proof}

We now turn to the proof of Lemma \ref{lemma:finishproof}. We will use the following result.

\begin{lemma} \label{lemma:y'}
    Let $G$ be a directed graph and let $P$ be a path of length $p$ starting at $v$ and ending at $v'$ with property \ref{prich} from Lemma \ref{lemma:findgoodpath}. Then there exists a subset $Y'\subset Y$ with $|Y\setminus Y'|\leq \gamma^{-2}|Y|+\gamma^{7}+6\gamma^{11}$ such that the following statements hold.
    
    \begin{enumerate}[label=(\alph*)]
        \item Let $S\subset Y'$ be a set of size less than $\gamma^{2}$ such that for any two distinct $y_i,y_j\in S$, we have $|i-j|>3\gamma^{9}$. Then $G$ has a path of length at least $p/2$ starting at $v$ and ending at $v'$ which only uses vertices from $V(P)\setminus S$.
        
        \item Let $z\in V(P)$ be a vertex between $y_h$ and $y_{h+1}$ on $P$. (If $z$ is before $y_1$, we take $h=0$ and if $z$ is after $y_{\ell}$, we take $h=\ell$.) Let $S\subset Y'$ be a set of size less than $\gamma^{2}$ such that for any $y_i\in Y'$ we have $|i-h|>\gamma^{9}$ and for any two distinct $y_i,y_j\in S$, we have $|i-j|>3\gamma^{9}$. Then $G$ has a path starting at $v$ and ending at $z$ which only uses vertices from $V(P)\setminus S$.
    \end{enumerate}
\end{lemma}

\begin{proof}
Let $t_0$ be the length of the subpath of $P$ from $v$ to $y_1$; for $1\leq i\leq \ell-1$, let $t_i$ be the length of the subpath of $P$ from $y_i$ to $y_{i+1}$; and let $t_{\ell}$ be the length of the subpath of $P$ from $y_{\ell}$ to $v'$.

Call $y_i$ \emph{crucial} if $\sum_{j:|j-i|\leq \gamma^{9}} t_j\geq \frac{p}{2\gamma^{2}}$. Let $y_{i_1},y_{i_2},\dots,y_{i_k}$ be the full list of crucial vertices. Then
$$\sum_{a=1}^k \sum_{j:|j-i_a|\leq \gamma^{9}} t_j \geq k\frac{p}{2\gamma^{2}}.$$
However, the left hand side is at most $(2\gamma^{9}+1)\sum_{j=0}^{\ell} t_j=(2\gamma^{9}+1)p$. Thus, $k\leq 2\gamma^{2}\cdot (2\gamma^{9}+1)\leq 6\gamma^{11}$.

Let $Y'$ be the subset of $Y$ consisting of those vertices $y_i$ that are not crucial and for which there exist $i-\gamma^{9}\leq j<i<j'\leq i+\gamma^{9}$ such that $y_jy_{j'}^+\in E(G)$. By property \ref{prich}, we have $|Y\setminus Y'|\leq \gamma^{-2}|Y|+\gamma^{7}+6\gamma^{11}$. It remains to prove (a) and (b).

\begin{enumerate}[label=(\alph*)]
    \item We will use the path $P$ with small modifications to make sure that we avoid the vertices in $S$. Let $S=\{y_{i_1},\dots,y_{i_s}\}$ where $i_1<i_2<\dots<i_s$. By the definition of $Y'$, for every $1\leq \alpha\leq s$, $G$ has an edge of the form $y_{j_{\alpha}}y_{j_{\alpha}'}^+$ for some $i_{\alpha}-\gamma^{9}\leq j_{\alpha}<i_{\alpha}<j_{\alpha}'\leq i_{\alpha}+\gamma^{9}$. So we can define the path $P'$ which is $P$ but with the subpath between each $y_{j_{\alpha}}$ and $y_{j_{\alpha}'}^+$ replaced by a direct edge. This is a valid path from $v$ to $v'$ since $y_{i_{{\alpha}+1}}-y_{i_{\alpha}}>3\gamma^{9}$ for every $\alpha$. It remains to verify that $P'$ has length at least $p/2$. Note that the length of $P'$ is precisely $p-\sum_{{\alpha}=1}^s \sum_{\beta=j_{\alpha}}^{j'_{\alpha}-1} t_{\beta}$. Since $y_{i_{\alpha}}\in S\subset Y'$, $y_{i_{\alpha}}$ is not crucial, so $$\sum_{\beta=j_{\alpha}}^{j_{\alpha}'} t_{\beta}\leq \sum_{\beta:|\beta- i_{\alpha}|\leq \gamma^{9}} t_{\beta}\leq \frac{p}{2\gamma^{2}}.$$ Since $s\leq \gamma^{2}$, $P'$ indeed has length at least $p/2$.
    
    \item The proof of this is very similar to that of (a). For any $y_i\in S$ with $i\leq h$, we just take an edge $y_jy_{j'}^+$ in $G$ with $i-\gamma^{9}\leq j<i<j'\leq i+\gamma^{9}$ and replace the path between $y_j$ and $y_{j'}^+$ with an edge to avoid $y_i$.
\end{enumerate}
\end{proof}

We are now ready to prove Lemma \ref{lemma:finishproof}.

\begin{proof}[Proof of Lemma \ref{lemma:finishproof}]

Since $c$ is small, we can assume that $d$ is sufficiently large. Suppose, for contradiction, that $G$ has no path of length at least $\phi(d)$ starting at $v$. Let $v'$ be the endpoint of $P$. Choose $Y'\subset Y$ provided by Lemma \ref{lemma:y'}. Note that $n\geq d^{10}$, otherwise $G$ contains a path of length $\phi(d)$ by Lemma~\ref{lemma:fewvertices} and Lemma \ref{lemma:shortcycle}.

Now we shall define several (not necessarily Eulerian) digraphs on vertex set $V(G)$. Call a path whose endpoints are on $P$ \emph{problematic} if its length is at most $d^{1/2}\gamma$ but it intersects $V(P)$ in more than $\gamma^{2}$ vertices. Let us keep removing problematic paths from $G$ as long as the remaining graph still has a problematic path. The process stops eventually; let $H_0$ be the resulting subgraph of $G$. Note that each removed path contains at least $\gamma^{2}$ edges incident to a vertex of $P$. There are at most $2pn$ such edges, so we have removed at most $2pnd^{1/2}\gamma^{-1}$ edges, i.e. $e(H_0)\geq e(G)-2pnd^{1/2}\gamma^{-1}$. Moreover, we removed at most $2pn\gamma^{-2}$ paths, so $\sum_{y\in Y'} |d^+_{H_0}(y)-d^-_{H_0}(y)|\leq 4pn\gamma^{-2}$. Also, $d^+_{H_0}(u)=d^-_{H_0}(u)$ for any $u\in V(G)\setminus V(P)$.

Now we define the digraphs $H_1,H_2,\dots$ recursively as follows. Assume we have already defined $H_i$. Then, if there exist distinct vertices $u,w\in V(G)\setminus V(P)$ such that $|N^+_{H_i}(u)\cap N^-_{H_i}(w)\cap Y'|\geq 7\gamma^{11}$ and $uw\not \in E(H_i)$, then define $H_{i+1}$ by adding the new edge $uw$ to $H_i$ and removing the edges $uy$ and $yw$ for some $y\in N^+_{H_i}(u)\cap N^-_{H_i}(w)\cap Y'$. If such vertices $u$ and $w$ do not exist, then terminate the algorithm and set $H=H_i$. Note that $\sum_{y\in Y'} |d^+_{H}(y)-d^-_{H}(y)|=\sum_{y\in Y'} |d^+_{H_0}(y)-d^-_{H_0}(y)|\leq 4pn\gamma^{-2}$ and $d^+_{H}(u)=d^-_{H}(u)$ for any $u\in V(G)\setminus V(P)$.
Moreover, this process has at most $pn$ steps, as there are at most $2pn$ edges touching the vertices of path $P$ and we remove two of them when going from $H_i$ to $H_{i+1}$.

We claim that there are at most $O(pn\gamma^{-2})$ edges in $H$ incident to a vertex of $P$. Indeed, any vertex in $V(P)\setminus Y$ has in- and outdegree at most $\gamma^{-2}n$ in $G$, and consequently also in $H$. Moreover, the number of vertices in $Y\setminus Y'$ is at most $\gamma^{-2}|Y|+\gamma^{7}+6\gamma^{11}\leq \gamma^{-2}p+7\gamma^{11}\leq 8\gamma^{-2}p$, where the last inequality follows from $p\geq d^{1/2}\gamma^{-5}$ and $\gamma^{11}\leq d^{1/2}\gamma^{-7}$. Hence, it suffices to prove that the number of edges in $H$ incident to a vertex in $Y'$ is at most $O(pn\gamma^{-2})$. Assume that this number is at least $24pn\gamma^{-2}$. Suppose that $\sum_{y\in Y'} d^+_{H}(y)\geq 12pn\gamma^{-2}$ (the case $\sum_{y\in Y'} d^-_{H}(y)\geq 12pn\gamma^{-2}$ is very similar). By $\sum_{y\in Y'} |d^+_{H}(y)-d^-_{H}(y)|\leq 4pn\gamma^{-2}$, we have $$\sum_{y\in Y': d^+_{H}(y)>2d^-_{H}(y)} d^+_{H}(y)\leq 8pn\gamma^{-2},$$
so
\begin{equation*}
    \sum_{y\in Y': d^+_{H}(y)\leq 2d^-_{H'}(y)} d^+_{H}(y)\geq 4pn\gamma^{-2}. \label{eqn:outdegs}
\end{equation*}
Hence, the number of walks $uyw$ in $H$ of length $2$ with $y\in Y'$ is
\begin{align*}
    \sum_{y\in Y'} d_H^-(y)d_H^+(y)
    &\geq \sum_{y\in Y': d_H^+(y)\leq 2d_H^-(y)} \frac{1}{2}d_H^+(y)^2\\
    &\geq \frac{1}{2}\cdot \frac{1}{|Y'|} \left(\sum_{y\in Y': d_H^+(y)\leq 2d_H^-(y)} d_H^+(y)\right)^2 \\
    &\geq \frac{1}{2p}(4pn\gamma^{-2})^2=8pn^2\gamma^{-4},
\end{align*}
where the second inequality follows from the Cauchy-Schwarz inequality. The number of walks $uyw$ in $H$ with $y\in Y'$ and $|N^+_{H}(u)\cap N^-_{H}(w)\cap Y'|<7\gamma^{11}$ is at most $n^2\cdot 7\gamma^{11}\leq 7pn^2\gamma^{-4}$ (since $p\geq d^{1/2}\gamma^{-5}$ and $\gamma^{20}= d^{1/2}$). So there exist at least $pn^2\gamma^{-4}$ walks $uyw$ in $H$ with $y\in Y'$ and $|N^+_{H}(u)\cap N^-_{H}(w)\cap Y'|\geq 7\gamma^{11}$, and hence, as $|Y'|\leq p$, there exist at least $n^2\gamma^{-4}$ pairs $(u,w)$ with $|N^+_{H}(u)\cap N^-_{H}(w)\cap Y'|\geq 7\gamma^{11}$. However, the number of pairs $(u,w)$ with $u\in V(P)$ or $w\in V(P)$ is at most $2pn$, and the number of pairs $(u,w)$ with $uw\in E(H)$ is at most $e(H)\leq nd$, so using $n\geq d^{10}$ it follows that there exist distinct $u,w\in V(G)\setminus V(P)$ such that $|N^+_{H}(u)\cap N^-_{H}(w)\cap Y'|\geq 7\gamma^{11}$ and $uw\not \in E(H)$. This contradicts the definition of~$H$.

Thus, there are indeed at most $O(pn\gamma^{-2})$ edges in $H$ incident to a vertex in $P$. Assume that $H$ has a path $Q$ of length $\phi(d)$ which starts at a vertex $z\in V(P)$ but contains no other vertex in $V(P)$. Let $u_1w_1,u_2w_2,\dots,u_sw_s$ be those edges of $Q$ which are not edges in $G$, taken in their natural order. Then for every $\alpha$, $|N^+_{H_0}(u_{\alpha})\cap N^-_{H_0}(w_{\alpha})\cap Y'|\geq 7\gamma^{11}$. Suppose that $s>\gamma^{2}$. Then it is easy to find greedily a path in $H_0$ of length at most $d^{1/2}\gamma$ which starts and ends in $V(P)$ and intersects $V(P)$ in more than $\gamma^{2}$ vertices. But this means that this path is problematic, contradicting the definition of $H_0$. Suppose now that $s\leq \gamma^{2}$. Choose $h$ such that $z$ is between $y_h$ and $y_{h+1}$ on $P$ (if $z$ is after $y_\ell$, then let $h=\ell$ and if $z$ is before $y_1$, then let $h=0$). Using $|N^+_{H_0}(u_{\alpha})\cap N^-_{H_0}(w_{\alpha})\cap Y'|\geq 7\gamma^{11}$, we may greedily choose $y_{i_{\alpha}}\in N^+_{H_0}(u_{\alpha})\cap N^-_{H_0}(w_{\alpha})\cap Y'$ for every $1\leq \alpha\leq s$ such that $|i_{\alpha}-h|>\gamma^{9}$ for every $1\leq \alpha\leq s$ and $|i_{\alpha}-i_{\beta}|>3\gamma^{9}$ for any two distinct $1\leq \alpha,\beta\leq s$. Then, by (b) in Lemma \ref{lemma:y'}, there exists a path from $v$ to $z$ which only uses vertices from $V(P)\setminus \{y_{i_1},\dots,y_{i_s}\}$. Concatenating this path with the path obtained from $Q$ by replacing the edges $u_{\alpha}w_{\alpha}$ with the paths $u_{\alpha}y_{i_{\alpha}}w_{\alpha}$, we obtain a path in $G$ of length at least $\phi(d)$ starting at $v$. Hence, we may assume that $H$ has no path of length $\phi(d)$ which starts at a vertex $z\in V(P)$ but contains no other vertex in $V(P)$.

Since $|N^+_{H}(u)|=|N^-_{H}(u)|$ for any $u\in V(H)\setminus V(P)$, we can partition the edge set of $H$ into a collection of two kinds of subgraphs: paths which start and end in $V(P)$ but have no internal vertices in $V(P)$, and cycles. Remove all edges from these paths and remove also all edges in cycles which intersect $V(P)$. Call the resulting graph $K$. The edge set of $K$ is a disjoint union of cycles, so $K$ is Eulerian. Since $H$ has no path of length $\phi(d)$ which starts at a vertex in $V(P)$ but contains no other vertex in $V(P)$, it follows that the number of removed edges from $H$ is at most $\phi(d)$ times the number of removed edges incident to $V(P)$. However, $H$ has at most $O(pn\gamma^{-2})$ edges incident to $V(P)$. Hence, $e(K)\geq e(H)-O(pn\gamma^{-2})\cdot \phi(d)\geq e(H)-O(pnd^{1/2}\gamma^{-1})$. Moreover, $e(H_{i+1})=e(H_i)-1$ for every $i$, so $e(H)\geq e(H_0)-pn$ (since as we explained above, the process has at most $pn$ steps). Together with $e(H_0)\geq e(G)-2pnd^{1/2}\gamma^{-1}$, we get $e(K)\geq e(G)-O(pnd^{1/2}\gamma^{-1})$. Since $e(G)\geq (n-1)d$, we have $e(K)\geq n(d-O(pd^{1/2}\gamma^{-1}))$. It follows that $K$ has a connected component $W$ such that $K[W]$ has average degree at least $d-Cpd^{1/2}\gamma^{-1}$ for some absolute constant $C$. As all vertices on $P$ are isolated in $K$, we have $W\subset V(G)\setminus V(P)$. Since $G-(V(P)\setminus \{v'\})$ is strongly connected, there exists a path $Q$ in $G-(V(P)\setminus \{v'\})$ from $v'$ to $W$ which only meets $W$ at its endpoint $z$. By induction, $K[W]$ has a path $R$ of length $\phi(d-Cpd^{1/2}\gamma^{-1})$ starting at $z$.

Note that $R$ is not necessarily a subgraph of $G$ since $K$ contains edges which are not present in $G$. However, we can turn $R$ into a path in $G$ with length at least $\phi(d-Cpd^{1/2}\gamma^{-1})$ as follows. Let $u_1w_1,u_2w_2,\dots,u_s w_s$ be those edges of $R$ which are not edges in $G$, taken in their natural order. Then for every $\alpha$, $|N^+_{H_0}(u_{\alpha})\cap N^-_{H_0}(w_{\alpha})\cap Y'|\geq 7\gamma^{11}$. Suppose that $s>\gamma^{2}$. Then (as above) it is easy to find a path in $H_0$ of length at most $d^{1/2}\gamma$ which starts and ends in $V(P)$ and intersects $V(P)$ in more than $\gamma^{2}$ vertices. But this means that this path is problematic, contradicting the definition of $H_0$. Suppose now that $s\leq \gamma^{2}$. Using $|N^+_{H_0}(u_{\alpha})\cap N^-_{H_0}(w_{\alpha})\cap Y'|\geq 7\gamma^{11}$, we may greedily choose $y_{i_{\alpha}}\in N^+_{H_0}(u_{\alpha})\cap N^-_{H_0}(w_{\alpha})\cap Y'$ for every $1\leq \alpha\leq s$ such that $|i_{\alpha}-i_{\beta}|>3\gamma^{9}$ for any two distinct $1\leq \alpha,\beta\leq s$. Let $R'$ be the path obtained from $R$ by replacing each $u_{\alpha}w_{\alpha}$ with the path $u_{\alpha}y_{i_{\alpha}}w_{\alpha}$.

By (a) in Lemma \ref{lemma:y'}, $G$ has a path of length at least $p/2$ starting at $v$ and ending at $v'$ which only uses vertices from $V(P)\setminus \{y_{i_1},y_{i_2},\dots,y_{i_s}\}$. Concatenating this path with $Q$ (which goes from $v'$ to $z$ and is internally disjoint from both $V(P)$ and $W$) and $R'$, we get a path in $G$ of length at least $p/2+\phi(d-Cpd^{1/2}\gamma^{-1})$ starting at $v$.
But by Lemma \ref{lemma:simpleineq},
$$\phi(d-Cpd^{1/2}\gamma^{-1})
    \geq \phi(d)-2cCp \geq \phi(d)-p/2$$
if $c$ is sufficiently small. This completes the proof.
\end{proof}

\section{Concluding remarks}

In this paper we proved that if $G$ is an Eulerian directed graph with average degree $d$, then it contains a path of length at least $cd^{1/2+\eps}$ for some absolute constant $\eps>0$ (we can take $\eps=1/40$). It would be interesting to extend this result to show that one can also find a cycle of length at least $cd^{1/2+\eps}$.

We also mention the following stronger version of Conjecture \ref{con:longcycle}.

\begin{conjecture}[Bollob\'as--Scott \cite{BS96}]
    Let $G$ be an Eulerian directed graph on $n$ vertices. Then the edge set of $G$ can be partitioned into $O(n)$ cycles.
\end{conjecture}

The difficulty of this conjecture is perhaps reflected by the fact that it is open even for simple graphs.

\begin{conjecture}[Erd\H os--Gallai \cite{Erd83,EGP66}]
    Let $G$ be a graph on $n$ vertices. Then the edge set of $G$ can be partitioned into $O(n)$ cycles and edges.
\end{conjecture}

The best bound for the last question is due to Conlon, Fox and Sudakov \cite{CFS14}, who showed that if $G$ is a graph on $n$ vertices, then the edge set of $G$ can be partitioned into $O(n\log \log n)$ cycles and edges.

\bigskip

\noindent
\textbf{Acknowledgements.} Oliver Janzer is supported by an ETH Zurich Postdoctoral Fellowship 20-1 FEL-35'. Benny Sudakov and Istv\'an Tomon are supported by the SNSF grant 200021\_196965. Istv\'an Tomon is also partially supported by the  Russian Government in the framework of MegaGrant no 075-15-2019-1926, and MIPT Moscow.

\end{document}